\documentclass[12pt]{amsart}

\usepackage[margin=3cm]{geometry}
\usepackage{amscd}
\usepackage{amsmath}
\usepackage{amssymb}
\usepackage{amsthm}
\usepackage{fancyhdr}
\usepackage{verbatim}

\usepackage[usenames, dvipsnames]{color}
\usepackage{tikz}
\usetikzlibrary{snakes}
\usepackage[all]{xy}
\usepackage[colorlinks=true, linkcolor=blue, citecolor=blue, pagebackref=true]{hyperref}
\usepackage[pagebackref=true]{hyperref}
\usepackage{mathrsfs}
\renewcommand{\mathcal}{\mathscr}
\renewcommand{\epsilon}{\varepsilon}
\renewcommand{\phi}{\varphi}
\numberwithin{equation}{section}

\theoremstyle{plain}
\newtheorem{theorem}{Theorem}[section]
\newtheorem{lemma}[theorem]{Lemma}
\newtheorem{remark}[theorem]{Remark}

\newtheorem{proposition}[theorem]{Proposition}

\providecommand{\floor}[1]{\lfloor#1\rfloor}

\providecommand{\ceil}[1]{\lceil#1\rceil}

\providecommand{\inner}[2]{\langle#1,#2\rangle}

\newcommand{\bp}{{\bf p}}

\newcommand{\bi}{{\overline \imath}}
\newcommand{\bk}{{\overline k}}
\newcommand{\bq}{{\overline q}}

\newcommand{\bj}{{\overline \jmath}}

\newcommand{\R}{\mathbb R}

\newcommand{\N}{\mathbb N}

\DeclareMathOperator{\Id}{Id}
\DeclareMathOperator{\diag}{diag}

\DeclareMathOperator{\fq}{freq}

\DeclareMathOperator{\freq}{freq}

\title{\textbf{Dimension of generic self-affine sets with holes}}
\author{Henna Koivusalo}
\address{University of Vienna\\ Oskar Morgensternplatz 1\\ 1090 Vienna, Austria}\email{henna.koivusalo@univie.ac.at}
\author{Micha\l\ Rams}
\address{Micha\l\ Rams\\Institute of Mathematics\\ Polish
Academy of Sciences\\ ul.
\'Sniadeckich 8, 00-656 Warszawa, Poland }\email{rams@impan.pl}
\date{\today}

\begin{document}

\thispagestyle{empty}

%===============================================

\begin{abstract}
Let $(\Sigma, \sigma)$ be a dynamical system, and let $U\subset \Sigma$. Consider the survivor set
\[
\Sigma_U=\{x\in \Sigma\mid \sigma^n(x)\notin U\textrm{ for all }n\}
\]
of points that never enter the subset $U$. We study the size of this set in the case when $\Sigma$ is the symbolic space associated to a self-affine set $\Lambda$, calculating the dimension of the projection of $\Sigma_U$ as a subset of $\Lambda$ and finding an asymptotic formula for the dimension in terms of the K\"aenm\"aki measure of the hole as the hole shrinks to a point. Our results hold when the set $U$ is a cylinder set in two cases: when the matrices defining $\Lambda$ are diagonal; and when they are such that the pressure is differentiable at its zero point, and the K\"aenm\"aki measure is a strong-Gibbs measure. 
\end{abstract}
\thanks{This project was supported by OeAD grant number PL03/2017. M.R. was supported by National Science Centre grant 2014/13/B/ST1/01033 (Poland).}

\maketitle

\section{Introduction}

Study of dynamical systems with holes begins from the question of \cite{PianigianiYorke79}: Assume you are playing billiards on table where trajectories of balls are unstable with respect to the initial conditions, and assume further, that a hole big enough for a ball to fall through is cut off the table. What is the asymptotic behaviour of the probability that at time $t$ a generic ball is inside some measurable set on the table, given that it is still on the table after time $t$? This and related questions have been studied in many dynamical systems, see \cite{ColletMartinezSchmitt94, ChernovMarkarianTroubetzkoy00, Demers05, BruinDemersMelbourne10} to name only few of many.

We will focus on a related problem of studying the set of those points that never enter the hole. To put this in rigorous terms, consider a continuous dynamical system $T: \Lambda\to \Lambda$ with a hole, the hole being an open subset $U\subset \Lambda$. Assume further, that there is an ergodic measure $\mu$ on $(T, \Lambda)$. How large is the survivor set,
\[
\Lambda_U=\{x\in \Lambda\mid T^n(x)\notin U\textrm{ for any }n\}?
\]
By Poincar\'e's recurrence theorem, this set will be of zero $\mu$-measure. Assuming that $\Lambda$ is a space where the notions of box-counting or Hausdorff dimension can be defined, we can continue by asking about the size of the survivor set in terms of its dimension. This set has also been studied in several contexts \cite{Urbanski89, LiveraniMaume03}, and in fact it turns out that, for example, the set of badly approximable points in Diophantine approximation can be written in terms of survivor sets under the iteration by the Gauss map \cite{Hensley92}.

The asymptotic speed at which the measure $\mu$ of the system escapes through the hole $U$ is the escape rate
\[
r_\mu(U)=-\lim_{n\to \infty}\tfrac 1n \log \mu\{x\in \Lambda\mid T^i(x)\notin U \textrm{ for }i<n\}
\]
(when the limit exists). In many systems the escape rate can be described in terms of the $\mu$-measure of the hole. In particular, often the escape rate and measure of the hole can also be used to quantify the asymptotic rate of decrease of the dimension deficit; that is, speed at which the dimension of the system with a hole approaches the dimension of the full system \cite{KellerLiverani09, BunimovichYurchenko11, FergusonPollicott12, Dettmann13}.

Recently, some interest has arisen in studying classical dynamical problems on self-affine fractal sets, under the dynamics that naturally arises from the definition of the set via an iterated function system \cite{KoivusaloRamirez, BaranyRams, FergusonJordanRams15}. (For definitions, see Section \ref{sec:setup}.) This is an interesting example  to consider since this dynamical system has an easy symbolic representation in terms of a full shift space, the dynamics of which is generally very well understood. In the presence of a separation condition the shift space is in fact conjugate to the dynamical system on the fractal set. However, in the affine case this dynamical system is not conformal. This means that a lot of the standard methodology cannot be carried through -- for example, the natural geometric potential is not in general multiplicative or commutative, and the dimension maximizing measure is not necessarily a Gibbs measure. In this article, as Theorems \ref{thm:diagonal} and \ref{thm:homogeneous}, we work out the asymptotic rate of decrease for the dimension deficit, for some classes of self-affine sets. As is to be expected from the historical point of view, the deficit is comparable to the measure of the hole, up to a constant which we quantify explicitly when possible.  Our proofs work in the case when the iterated function system consists of diagonal matrices (Theorem \ref{thm:diagonal}, for a simpler corollary see Theorem \ref{thm:main}) and in the case when the pressure corresponding to the iterated function system has a derivative at its zero point, and the K\"aenm\"aki measure is a strong-Gibbs measure (Theorem \ref{thm:homogeneous}, for definitions see Section \ref{sec:setup}). 

\section{Problem set-up and notation}\label{sec:setup}

Let $\{A_1,\dots, A_k\}$ be a finite set of contracting non-singular $d\times d$ matrices, and let $(v_1,\dots,v_k)\in\R^d$. Consider $\{f_1, \dots, f_k\}$, the iterated function system (IFS) of the affine mappings $f_i:\R^d\to \R^d, f_i(x)=A_i(x)+v_i$ for $i=1, \dots, k$. It is a well known fact that there exists a unique non-empty compact subset $\Lambda$ of $\R^d$ such that
\begin{equation}\label{eq:definv}
\Lambda=\bigcup_{i=1}^kf_i(\Lambda).
\end{equation}

This set has a description in terms of the shift space. Let $\Sigma$ be the set of one-sided words of symbols $\left\{1,\dots,k\right\}$ with infinite length, i.e.\ $\Sigma=\left\{1,\dots,k\right\}^{\N}$, and $\Sigma_n=\{1, \dots, k\}^n$. Let us denote the left-shift operator on $\Sigma$ by $\sigma$. When applied to a finite word $\bi\in \Sigma_n$, $\sigma(\bi)=i_2 \dots i_n$, the word of shorter length with the first digit deleted. Let the set of words with finite length be $\Sigma^*=\bigcup_{n=0}^{\infty}\Sigma_n$ with the convention that the only word of length $0$ is the empty word. Denote the length of $\bi \in \Sigma^*$ by $|\bi|$, and for finite or infinite words $\bi, \bj$, let $\bi\wedge\bj$ denote their joint beginning. If $\bi$ can be written as $\bi=\bj\bk$ for some finite or infinite word $\bk$, denote $\bj<\bi$. We define the cylinder sets of $\Sigma$ in the usual way, that is, by setting $[\bi]=\left\{\bj\in\Sigma: \bi<\bj\right \}$ for $\bi\in \Sigma^*$. For a word $\bi=(i_1,\dots,i_n)$ with finite length let $f_{\bi}$ be the composition $f_{i_1}\circ\cdots\circ f_{i_n}$ and $A_{\bi}$ be the product $A_{i_1}\cdots A_{i_n}$. For $\bi\in\Sigma^*\cup \Sigma$, denote by $\bi|_n$ the first $n$ symbols of $\bi$, i.e. $\bi|_n=(i_1,\dots,i_n)$. We define $\bi|_0=\emptyset$, $A_{\emptyset}=\Id$, the identity matrix, and $f_{\emptyset}=\Id$, the identity function.

We define a \textit{natural projection} $\pi:\Sigma \to \Lambda$ by
\[
\pi(\bi)=\sum_{k=1}^{\infty}A_{\bi|k-1}v_{i_k},
\]
and note that $\Lambda=\cup_{\bi\in \Sigma}\pi(\bi)$.

Denote by $\sigma_i(A)$ the $i$-th {\it singular value} of a matrix $A$, i.e. the positive square root of the $i$-th eigenvalue of $AA^*$, where $A^*$ is the transpose of $A$. We note that $\sigma_1(A)=\|A\|$, and $\sigma_2(A)=\|A^{-1}\|^{-1}$, where $\|\cdot\|$ is the usual matrix norm induced by the Euclidean norm on $\R^d$. Moreover, $|\sigma_1(A)\cdots\sigma_d(A)|=|\det A|$. For $s\ge 0$ define the {\it singular value function} $\phi^s$ as follows
\begin{equation}\label{esvf}
\phi^s(A):=\sigma_1\cdots \sigma_{\ceil{s}}^{s-\floor{s}},
\end{equation}
where $\ceil{\cdot}$ and $\floor{\cdot}$ are the ceiling and floor function. Further, for an affine IFS, define the {\it pressure function}
\begin{equation}\label{eq:pressure}
P(s)=\lim_{n\to \infty}\frac 1n\log \sum_{\bi\in \Sigma_n}\phi^s(A_\bi).
\end{equation}
When it is necessary to make the distinction, we will write $P(s, (A_1, \dots, A_k))$. Given a measure $\nu$ on $\Sigma$, we define the {\it entropy}
\[
h_\nu=-\lim_{n\to\infty} \frac1n \sum_{i\in \Sigma_n}\nu[\bi]\log\nu[\bi],
\]
and {\it energy}
\[
E_\nu(t)=\lim_{n\to \infty}\frac 1n\sum_{\bi\in \Sigma_n}\nu[\bi]\log\phi^t(A_\bi).
\]
By a result of K\"aenm\"aki \cite{Kaenmaki04}, for all $s\ge 0$ {\it equilibrium} or \emph{K\"aenm\"aki measures} exist, that is, for all $s$ there is a measure $\mu=\mu(s)$ on $\Sigma$ such that
\[
P(s)=E_\mu(s)+h_\mu.
\]
A classical result of Falconer \cite{Falconer88} (see also \cite{Solomyak98}) asserts that when $\|A_i\|<1/2$, for almost all $(v_1, \dots, v_k)\in \R^{dk}$, the dimension of the self-affine set $\Lambda$ is given by the $s$ for which $P(s)=0$ (or $d$ if the number $s$ is greater than $d$), and K\"aenm\"aki proves that $\dim \Lambda=\dim \mu$ for the equilibrium measure at this value of $s$. %We will call the K\"aenm\"aki measure $\mu$ {\it geometric} if for some $c<C$ and all $\bi$ we have
%\[
%c \mu[\bi] \leq \phi^s(A_\bi) \leq C \mu[\bi].
%\]

We will need the notion of a {\it Bernoulli measure}, that is, given a probability vector $(p_1, \dots, p_k)$ the Bernoulli measure $\bp$ is the probability measure on $\Sigma$ giving the weight $p_\bi=p_{i_1}\cdots p_{i_n}$ to the cylinder $[\bi]$. We will also need the notion of an \emph{$s$-semiconformal measure}, that is, a measure $\mu$ for which constants $0<c\le C<\infty$ exist such that for all $\bi\in \Sigma^*$,
\[
ce^{-|\bi|P(s)}\phi^s(A_\bi)\le \mu([\bi])\le Ce^{-|\bi|P(s)}\phi^s(A_\bi).
\]
In this terminology we are following \cite{KaenmakiVilppolainen10}, where the existence of such measures for an affine iterated function system is investigated. %We note that in the literature $s$-semiconformal measures are sometimes called Gibbs measures. However, in the wider context of dynamical systems, Gibbs measures are usually assumed to satisfy the above semiconformal property for a multiplicative potential. To make sure that there is no danger of misunderstandings, w
We call an $s$-semiconformal measure $\mu$ a \emph{strong-Gibbs measure}, if it is both $s$-semiconformal and also a Gibbs measure for some multiplicative potential. 
We now define the survivor sets we will be interested in. Fix some $\bq\in \Sigma_q$ and let $U=[\bq]$. In the symbolic space $\Sigma$ we define the \emph{survivor set }as
\[
\Sigma_U=\{\bi\in \Sigma\mid \sigma^i(\bi)\notin U\textrm{ for all }i\}.
\]
Whenever it is the case that $f_i(\Lambda)\cap f_j(\Lambda)=\emptyset$ for $i\neq j$, then it is possible to define a dynamical system $T:\Lambda\to \Lambda$ such that for $x\in f_i(\Lambda)$ we let $T(x)=f_i^{-1}(x)$. In this case it is also true that the projection map $\pi$ is a bijection, and the dynamical system $(\Lambda, T)$ is conjugate to the full shift $(\Sigma, \sigma)$, that is, $\pi\circ\sigma=T\circ \pi$. Hence in this case the survivor sets in the symbolic space $(\Sigma, \sigma)$ and on the fractal $(\Lambda, T)$ correspond to each other, that is,
\[
\pi(\Sigma_{[\bq]})=\{x\in \Lambda \mid T^i(x)\notin \pi([\bq])\textrm{  for all }i\}.
\]
This is why we define, also in the general situation, the \emph{survivor set }on $\Lambda$ to be $\Lambda_{\pi[\bq]}=\pi(\Sigma_{[\bq]})$. In the following we will be interested in the dimension of the set $\pi(\Sigma_{[\bq]})$, regardless of whether or not the projection $\pi$ is bijective and the dynamics $T$ well-defined.

We can now formulate our main theorems concerning the Hausdorff dimension of the survivor set. In the following the point $\bq$ will be fixed and it will cause no danger of misunderstanding to denote, $\Lambda_{\pi[\bq|_q]}=\Lambda_q$, where $q$ is a positive integer. In Section \ref{sec:diagonal}, as Theorem \ref{thm:diagonal}, we prove a statement for diagonal matrices. However, the formulation of the theorem in the diagonal case requires technical notation that we want to postpone introducing. For the full statement we refer the reader to Theorem \ref{thm:diagonal}, here we only give the special case where the diagonal elements are in the same order.
\begin{theorem}\label{thm:main}
Let $\Lambda$ be a self-affine set corresponding to an iterated function system $\{A_1+v_1, \dots, A_k+v_k\}$ with $\|A_i\|<\tfrac 12$ for all $i=1, \dots, k$, and let $\bq\in \Sigma$. Assume that $A_i=\diag(a^i_1, \dots, a^i_d)$ are diagonal for all $i=1, \dots, k$, and, furthermore, that the diagonal elements are in the same increasing order $a^i_1\leq \ldots \leq a^i_d$ in all of the matrices. Denote by $\mu$ the K\"aenm\"aki measure for the value of $s$ for which $P(s)=0$. Then, for Lebesgue almost all $(v_1, \dots, v_k)\in \R^{dk}$,
\[
\lim_{q\to \infty}\frac{\dim\Lambda - \dim\Lambda_{q}}{\mu[\bq|_q]}=\begin{cases}
\frac{1}{Z}, & \bq\textrm{ is not periodic }\\
 \frac{1 - \mu[\bq|_\ell]}{Z}, & \bq \textrm{ is periodic with period }\ell,
\end{cases}
\]
where the explicit constant $Z$, which only depends on the diagonal elements of the matrices $A_i$, is defined in Remark \ref{rem:exs2}.
\end{theorem}

\begin{theorem}\label{thm:homogeneous}
Let $\Lambda$ be a self-affine set corresponding to an iterated function system $\{A_1+v_1, \dots, A_k+v_k\}$ with $\|A_i\|<\tfrac 12$ for all $i=1, \dots, k$, and let $\bq\in \Sigma$. Denote by $\mu$ the K\"aenm\"aki measure for the value of $s$ for which $P(s)=0$, assume also that $P$ is differentiable at this point. Assume that $\mu$ is a strong-Gibbs measure - in particular, a Gibbs measure for a multiplicative potential $\psi$. Then, for Lebesgue almost all $(v_1, \dots, v_k)\in \R^{dk}$,
\[
\lim_{q\to \infty}\frac{\dim\Lambda - \dim\Lambda_{q}}{\mu[\bq|_q]}=\begin{cases}
-\frac{1}{P'(s)}, & \textrm{ when }\bq\textrm{ is not periodic }\\
 -\frac{1 - \psi(\bq|_\ell)}{P'(s)}, & \textrm{ when }\bq \textrm{ is periodic with period }\ell.
\end{cases}
\]
\end{theorem}

This theorem will be proved in Section \ref{sec:abstract}. 

\begin{remark}
\begin{itemize}
\item[(1)] It might be tempting to think that, since the result above holds for diagonal matrices it would be easy to extend it to the case of upper triangular matrices. The temptation is due to \cite[Theorem 2.5]{FalconerMiao07}, which states that for an iterated function system with upper triangular matrices the pressure only depends on the diagonal elements of the matrices. However, this does not seem straightforward, see Remark \ref{rem:exs3}.
\item[(2)] Notice that in the statements of Theorems \ref{thm:homogeneous} and \ref{thm:main} the normalizing factor in the denominator of the limit plays the same role as the Lyapunov exponent in, for example, \cite{FergusonPollicott12}.
\end{itemize}
\end{remark}

\section{The pressure formula for the dimension and other facts}

From here on we consider the point $\bq\in \Sigma$ fixed, and denote $\Lambda_{\pi[\bq|_q]}=\Lambda_q$ for a choice of positive integer $q$. We start by recalling a pressure formula for the dimension of the surviving set.

Denote, for $n\in \N$,
\[
\Sigma_{n,q}=\{\bi|_n\in \Sigma_n\mid \sigma^i(\bi)\notin[\bq|_q]\textrm{ for all }i\}.
\]
Define the {\it reduced pressure }
\[
P_q(t)=\lim_{n\to \infty}\frac 1n \log \sum_{\bi\in \Sigma_{n, q}}\phi^t(A_\bi).
\]
\begin{theorem} \label{thm:dimform}
Let $\bq\in \Sigma$, $q \in \N$. For an iterated function system $\{A_1+v_1, \dots, A_k+v_k\}$ with $\|A_i\|<\tfrac 12$, for Lebesgue almost all $(v_1, \dots, v_k)\in \R^{dk}$,
\[
\dim \Lambda_{q}=\min\{d, t_q\}
\]
where $t_q$ is the unique value for which $P_q(t_q)=0$.
\end{theorem}
\begin{proof}
This is \cite[Theorem 5.2]{KaenmakiVilppolainen10}.
\end{proof}
\begin{remark}\label{rem:close dim}
Notice that as $q\to \infty$, the reduced pressure approaches the full pressure, and hence the dimension of the surviving set $\Lambda_{q}$ approaches the dimension of $\Lambda$.
\end{remark}
\begin{remark}
The set $\Sigma_{n,q}$ can be written in an equivalent form
\[
\Sigma_{n,q}=\{\bi|_n\in \Sigma_n\mid \bi\in \Sigma\textrm{ is such that }\sigma^{i}(\bi)\notin[\bq|_q]\textrm{ for all }i<n\},
\]
since any point that does not enter the hole in the first $n$ iterations can be completed to a word that never enters the hole.
\end{remark}

The following facts about the K\"aenm\"aki measure are standard.
\begin{lemma}\label{lem:measure}
Consider the K\"aenm\"aki measure $\mu$ at the value $s_0$, where $s_0$ is the root of $P$.
\begin{itemize}
\item[a)] When there is some $A$ such that $A_i=A$ for all $i=1, \dots, k$, then $\mu$ is the Bernoulli measure with equal weights.
\item[b)] When $A_i$ are diagonal matrices with the size of the diagonal elements in the same order, then $\mu$ is a Bernoulli measure with cylinder weights $\phi^{s_0}(A_1), \dots, \phi^{s_0}(A_k)$.
\end{itemize}
\end{lemma}

\begin{proof}
\begin{itemize}
\item[a)] Immediate.
\item[b)] In the diagonal case the singular value function is multiplicative. Hence the zero of the pressure is obtained at the point where $\sum_{i=1}^d\phi^s(A_i)=1$, so that $\phi^{s_0}(A_i)$ define a probability vector. The K\"aenm\"aki measure is a Bernoulli measure with these weights.
\end{itemize}
\end{proof}

%The assumptions of Theorem \ref{thm:homogeneous} may look difficult to satisfy, but there are at least two classes of systems for which the K\"aenm\"aki measure is strong-Gibbs.

%The first is given by the new result of B\'ar\'any, K\"aenm\"aki, and Morris \cite{BKM}:
%\begin{prop} \label{prop:example}
%If $d=2$ and the matrices $A_i$ form a dominated cocycle then the K\"aenm\"aki measure is strong-Gibbs.
%\end{prop}
%The dominated cocycles are an open subset of $GL(2,\R)$-cocycles, we refer the reader to \cite{BKM} for the definition.
%The other example is more obvious: the same is true if, for any $d$, all the matrices $A_i$ are of the form $A_i = A^{k_i}$ for some matrix $A$. 

Define the {\it escape rate} of a measure $\nu$ on $\Sigma$ as
\[
r_{\nu}(U)=-\lim_{n\to \infty}\frac 1n \log\nu\{\bi\in \Sigma\mid \sigma^i(\bi)\notin U\textrm{ for }i<n\},
\]
when the limit exists. We quote the following special case of Ferguson and Pollicott \cite{FergusonPollicott12}. In the theorem we make a reference to $P(\psi)$, the pressure corresponding to a potential $\psi$. This is defined analogously to $P(s)$ in \eqref{eq:pressure}, but with $\psi$ in place of $\phi^s$. We note that we will, in fact, only apply Theorem \ref{thm:ferguson pollicott} when $P(\psi)=P(s)$.

\begin{theorem}\label{thm:ferguson pollicott}
Let $\bq\in \Sigma$ and let $U_q=[\bq|_q]$. Consider a multiplicative potential $\psi$ for which $P(\psi)=0$. For a Gibbs measure $\mu$ on $\Sigma$, the escape rate $r_\mu(U_q)$ always exists and
\[
\lim_{q\to \infty} \frac{r_\mu(U_q)}{\mu(U_q)}=
 \begin{cases}
   1, &\textrm{ if } \bq\textrm{ is not periodic }\\
    1 - \psi(\bq|_\ell), &\text{ if } \bq \textrm{ is periodic with period }\ell.
  \end{cases}
\]
\end{theorem}
\begin{proof}
See \cite[Proposition 5.2 and Theorem 1.1]{FergusonPollicott12}  or see \cite[Theorem 2.1]{KellerLiverani09}.
\end{proof}
Notice that in order for us to apply this theorem in our set-up it is essential that the measure $\mu$ is also $s$-semiconformal.
\begin{lemma} \label{lem:onlypressure}
Let $\bq\in \Sigma$ and let $U_q=[\bq|_q]$. Let $s$ be where the pressure $P(s)=0$. Let $\mu$ be the K\"aenm\"aki measure at this value $s$, and assume that it is a strong-Gibbs measure, in particular, Gibbs for some multiplicative potential $\psi$. Then
\[
\lim_{q\to \infty}\frac{P(s) - P_{q}(s)}{\mu(U_q)}=
 \begin{cases}
   1, &\textrm{ if } \bq\textrm{ is not periodic }\\
    1 - \psi(\bq|_\ell), &\text{ if } \bq \textrm{ is periodic with period }\ell.
  \end{cases}
\]
\end{lemma}
\begin{proof}
%By Lemma \ref{lem:measure} the K\"aenm\"aki measure is a Bernouli measure with the probability vector $(\phi^s(A_1), \dots, \phi^s(A_k))$. Hence,
We have, using the $s$-semiconformal property
\begin{align*}
P(s) - P_{q}(s)&=0 - \lim_{n\to \infty} \tfrac 1n \log \sum_{\bi\in \Sigma_{n,q}}\phi^s(A_\bi) \\
&=- \lim_{n\to \infty} \tfrac 1n \log \sum_{\bi\in \Sigma_{n,q}}\mu[\bi]\\
&=r_{\mu}([\bq|_q]).
\end{align*}
The proof is now finished by Theorem \ref{thm:ferguson pollicott}.
\end{proof}

\section{Diagonal matrices}\label{sec:diagonal}

Let us start from a more detailed description of the singular value pressure in the diagonal case. Let $D=(e_1, \ldots, e_d)\in S_d$ be a permutation of $\{1,\ldots,d\}$. For a diagonal matrix $A=\diag(a_j)$ denote

\[
\varphi_D^s(A) = a_{e_1} \cdot \ldots \cdot a_{e_{\lfloor s \rfloor}} \cdot a_{e_{\lfloor s \rfloor +1}}^{s-\lfloor s \rfloor}.
\]

Naturally,
\[
\varphi_D^s(A) \leq \varphi^s(A) \leq \sum_D \varphi_D^s(A).
\]
Hence, if we define the $D$-pressure analogously to the singular value pressure
\[
P_D(s) = \lim_{n\to\infty} \frac 1n \log \sum_{\bi\in \Sigma_n}\varphi_D^s(A_\bi)
\]
and the reduced $D$-pressure analogously to the reduced pressure
\[
P_{D,q}(s) = \lim_{n\to\infty} \frac 1n \log \sum_{\bi\in \Sigma_{n,q}}\varphi_D^s(A_\bi)
\]
then
\[
P(s) = \max_{D\in S_d} P_D(s), P_q(s) = \max_{D\in S_d} P_{D,q}(s).
\]
In particular, denoting by $t_q^D$ the zero of $P_{D,q}$, we have
\[
\dim \Lambda_q = \min \{d, \max_D t_q^D\}
\]
whenever the assumptions of Theorem \ref{thm:dimform} are satisfied.

Thus, in order to find the zero of $P_q$ it will be enough for us to be able to find the zeroes $t_q^D$ for all choices of $D$. Which will be significantly simplified by the fact that, contrary to $\varphi^s$, $\varphi_D^s$ is a multiplicative potential. Moreover, to prove Theorem \ref{thm:main} we do not need to check all possible $D$: as $P_{D,q}\to P_D$ when $q\to\infty$, it is enough for us to only consider those $D$ for which $P_D(s_0)=P(s_0)=0$.

Let us start by denoting by $\mu_D$ the Bernoulli measure with the probability vector $(p_1^D,\ldots, p_k^D)=(\varphi_D^{s_0}(A_1), \ldots, \varphi_D^{s_0}(A_k))$. Because $\varphi_D^{s_0}$ is multiplicative, as in Lemma \ref{lem:measure} we see that this really is a probability vector. Observe that even though we only consider $D$ for which $P_D(s_0)=0$, this measure can still in general depend on $D$.

%Consider the case of diagonal matrices, where the diagonal elements $(a_1^i, \dots ,a_d^i)$ of $A_i$ are in the same size order for all $i$. Denote the K\"aenm\"aki measure by $\mu$.

Recall Lemma \ref{lem:onlypressure}, and notice that by the multiplicativity of the potential $\phi^s_D$, the proof of Lemma \ref{lem:onlypressure} goes through unaltered for $\mu_D$, the $D$-pressure and reduced $D$-pressure. Furthermore, $\mu_D$ is a Gibbs measure for the potential $\phi^s_D$. 

The idea of the proof of Theorem \ref{thm:main} is as follows. We fix some $D$ for which $P_D(s_0)=0$ and then we will bound $s_0-t_q^D$ from above and below with bounds, the difference between which approaches 0 faster than $-P_{D,q}(s_0)$ as $q\to\infty$. This will let us estimate the limit of $(s_0-t_q^D)/\mu_D([\bq|_q])$. To simplify the notation, we will skip the index $D$ in the rest of this section -- but the reader should remember that the potential $\varphi^s$ we work with is not the singular value function but an auxiliary multiplicative potential which is only equal to the singular value function in the case when the diagonal elements $(a_1^i,\ldots,a_d^i)$ are in the same order for all $i$.

We need some notation. Denote by $\Delta$ the simplex of length $k$ probability vectors. Given a finite word $\bi\in \Sigma_n$, let
\[
\fq(\bi)=\frac 1n (\#\{i\in \{1, \dots, n\}\mid \bi_i=1\}, \dots, \#\{i\in \{1, \dots, n\}\mid \bi_i=k\})\in \Delta,
\]
and for an infinite word $\bi\in \Sigma$,
\[
\fq(\bi)=\lim_{n\to \infty} \frac 1n\freq(\bi|_n)\in \Delta,
\]
if the limit exists. Fix $\epsilon>0$. Let $E=E(\epsilon)$ be $\epsilon$-dense in $\Delta$. Then the number of elements of $E$, $\#E=\epsilon^{1-k}=:N$. Fix $\alpha\in \Delta$, and denote
\[
F_n(\alpha)=\{\bi\in \Sigma_n\mid \max_i|\fq(\bi)(i) - \alpha(i)|<\epsilon\},
\]
where $\alpha(i)$ is the $i$-th coordinate of $\alpha$ and the same for $\fq(\bi)$. Assume, without loss of generality, that $E$ was chosen in such a way that every point in $\Sigma_n$ belongs to at most some $K_d$ of the sets $F_n(\alpha)$, where $\alpha\in E(\varepsilon)$, with the constant $K_d$ not depending on $n$.

Further, given some $\alpha\in \Delta$, denote
\[
A(\alpha)=\diag((a_{e_1}^1)^{\alpha(1)}\cdots (a_{e_1}^k)^{\alpha(k)}, \dots, (a_{e_d}^1)^{\alpha(1)}\cdots (a_{e_d}^k)^{\alpha(k)})
\]
This is a kind of a dummy matrix simulating the frequency $\alpha$. Finally, let $o(\epsilon)$ be a function that approaches $0$ as $\epsilon \to 0$, and $o(n)$ a function that approaches $0$ as $n\to \infty$.

\begin{lemma}\label{lem:approximation}
At a given scale we can approximate $P(s)$ by sequences of only one frequency; that is, given $\epsilon>0$ and $n>0$, there is $\alpha\in E(\epsilon)$ such that the numbers

\[
\frac 1n \log \sum_{\bi \in \Sigma_n} \varphi^s(A_\bi),\quad \frac 1n \log \sum_{\bi \in F_n(\alpha)} \phi^s(A_\bi),\textrm{ and}\quad\frac 1n \log \sum_{\bi \in F_n(\alpha)} \phi^s(A(\alpha))^n
\]
are all $o(\varepsilon, n)$-close to each other. The same statement holds when we restrict all these sums to $\Sigma_{n,q}$.
\end{lemma}
\begin{proof}
Fix $\epsilon>0$ and $n>0$. Notice that, when $|\alpha-\freq (\bi)|< \epsilon$ for $\bi\in \Sigma_n$, then
\begin{equation}\label{eq:freq}
c_1^{\epsilon n}\phi^s(A(\alpha))^n\le \phi^s(A_\bi)\le c_2^{\epsilon n}\phi^s(A(\alpha))^n,
\end{equation}
for constants $c_1, c_2>0$ that do not depend on $n$ and $\epsilon$. Furthermore, for all $\alpha \in E$
\[
\sum_{\bi\in F_n(\alpha)}\phi^s(A_\bi) \leq \sum_{\bi\in \Sigma_n}\phi^s(A_\bi) \leq \sum_{\alpha\in E}\sum_{\bi\in F_n(\alpha)}\phi^s(A_\bi).
\]
As $E$ is a finite set, there exists $\alpha$ for which
\[
\sum_{\bi\in F_n(\alpha)}\phi^s(A_\bi)\ge \frac {1}{\# E}\sum_{\bi\in \Sigma_n}\phi^s(A_\bi)
\]
and we are done. The proof for sums restricted to $\Sigma_{n,q}$ instead of $\Sigma_n$ is exactly the same.
\end{proof}

Fix $\epsilon>0$ and $n>0$. Define $\tilde{g}^s, g^s_q:\Delta\to \mathbb R$ by setting for all $\alpha\in \Delta$
\[
\tilde{g}^s(\alpha) = \frac 1n \log \sum_{\bi\in F_n(\alpha)} \phi^s(A_\bi)
\quad\textrm{and}\quad
g^s_q(\alpha) = \frac 1n \log \sum_{\bi\in F_n^q(\alpha)} \phi^s(A_\bi)
\]
where $F_n^q(\alpha)\subset \Sigma_{n,q}$ is defined analogously to $F_n(\alpha)$. Further, for $\alpha \in \Delta$, denote
\[
g^s(\alpha) = f(\alpha) + \inner{a(s)}{ \alpha}
\]
for $f(\alpha) = - \sum_{i=1}^k \alpha(i) \log \alpha(i)$ and
\[
a(s)=(\log (a_{e_1}^1\cdots (a_{e_{\ceil {s}}}^1)^{s-\floor{s}}, \dots, \log (a_{e_1}^k\cdots (a_{e_{\ceil {s}}}^k)^{s-\floor{s}}).
\]
By virtue of \eqref{eq:freq}, for $n$ large,
\begin{equation}\label{eq:tilde g}
|\tilde{g}^s(\alpha) - g^s(\alpha)|<o(\epsilon, n).
\end{equation}
%\aside{Why is $\#F_n(\alpha)=\exp(n(-\sum_i\alpha_i\log\alpha_i)+o(\epsilon))$?}
Given $s\ge 0$, denote by $\alpha^s$ the point of $\Delta$ where $g^s$ achieves maximum, and by $\alpha_q^s$ the point (or one of the points, if it is not unique) of $\Delta$ where $g^s_q$ achieves maximum. Observe that those are (almost exactly) the maximizing frequencies given by Lemma \ref{lem:approximation}. Indeed, for the latter this is obvious, while for the former we have $\#F_n(\alpha)=\exp(n(-\sum_i\alpha_i\log\alpha_i)+o(\epsilon))$, hence maximizing $g^s$ means (almost) maximizing the sum $\sum_{\bi\in F_n(\alpha)} \varphi^s(A_\bi)$.

\begin{lemma} \label{lem:concave}
For any $s, t$, there exists $w=w_s>0$ depending on only one of the parameters, such that
\[
|\alpha^{s}-\alpha^{t}| \leq \frac {|a(s)-a(t)|} {2w}
\]
and
\[
g^{s}(\alpha^{s}) \ge g^{s}(\alpha^{t}) + \frac {|a(s)-a(t)|^2} {4w}.
\]
\end{lemma}
\begin{proof}
Note that as a function of $\alpha$, the function $g^s: \Delta\to \R$ is strictly concave for every $s$, so that there exists a number $w=w_s>0$ such that for the second differential in direction $e$,
\[
\inf_{\alpha, e} D^2_e g^s(\alpha) \leq -2w <0.
\]
That means that
\begin{equation} \label{eqn:concave}
g^s(\alpha) \leq g^s(\alpha^s) - w|\alpha-\alpha^s|^2.
\end{equation}
Next fix $t$ and $s$ and notice that
\begin{align*}
g^t(\alpha^t)= f(\alpha^t)+\inner{a(t)}{\alpha^t}&=g^s(\alpha^t)+\inner{(a(s) - a(t))}{\alpha^t}\\
&\le g^s(\alpha^s) - w|\alpha^t - \alpha^s|^2+\inner{(a(t) - a(s))}{\alpha^t}.
\end{align*}
If the first claim does not hold, that is, $|\alpha^t - \alpha^s|> |a(s) - a(t)|/(2w)$, we obtain from the above
\[
g^t(\alpha^t)<g^s(\alpha^s) - \frac{|a(s) - a(t)|^2}{4w}+\inner{(a(s) - a(t))}{\alpha^t}<g^t(\alpha^s),
\]
which is a contradiction with the maximality of $\alpha^t$. The second claim is immediate from here.
\end{proof}

\begin{lemma}\label{lem:com}
There is a constant $L$ such that for all $s, t\ge 0$,
\[
g^s(\alpha^s) - g^t(\alpha^t)\le L|a(s) - a(t)|.
\]
\end{lemma}
\begin{proof}
By the definition of $g^s$ and compactness of $\Delta$, there is an $L$ such that
\[
g^{s}(\alpha^s) - g^t(\alpha^s)=\inner{a(s) - a(t)}{\alpha^s}\le L|a(s) - a(t)|.
\]
Furthermore, by maximality of $\alpha^t$,
\[
g^s(\alpha^s) - g^t(\alpha^s)\ge g^s(\alpha^s) - g^t(\alpha^t).
\]
\end{proof}

\begin{lemma} \label{lem:aprox}
The functions $g^s$ and $g^s_q$ are good approximations to $P(s)$ and $P_q(s)$. That is,
\[
P(s) = g^s(\alpha^s) + o(\epsilon, n) \quad\textrm{and}\quad P_q(s) = g^s_q(\alpha^s_q) + o(\epsilon, n) .
\]
\end{lemma}
\begin{proof}
The second part of the assertion follows from
\[
\max_\alpha e^{n g^s_q(\alpha)} \leq \sum_{\bi\in \Sigma_q^n} \phi^s(T_\bi) \leq \sum_{\alpha\in E(\epsilon)} e^{n g^s_q(\alpha)} \leq \epsilon^{1-k} \cdot \max_{\alpha\in \Delta} e^{n g^s_q(\alpha)}.
\]
This calculation also applies to $\tilde g^s$, and by \eqref{eq:tilde g} $g^s$ can be approximated $o(\epsilon, n)$-closely by $\tilde g^s$.
\end{proof}

\begin{lemma} \label{lem:delta0} Let $s_0$ satisfy $P(s_0)=0$. The distance between the frequencies maximizing $g^{s_0}$ and $g^{s_0}_q$ is controlled by $P_q(s_0)$. That is,
\[
|\alpha^{s_0} - \alpha^{s_0}_q| \leq \left(\frac {-P_q(s_0)} {w_{s_0}}\right)^{1/2} + o(\epsilon, n).
\]
\end{lemma}
\begin{proof}
Notice that by Lemma \ref{lem:aprox} and \eqref{eqn:concave}
\begin{align*}
g^{s_0}_q(\alpha^{s_0}_q)&=P_q(s_0)+o(\epsilon, n)=P(s_0)+P_q(s_0)+o(\epsilon, n)\\
&=g^{s_0}(\alpha^{s_0})+P_q(s_0)+o(\epsilon, n)\\
&\ge g^{s_0}(\alpha^{s_0}_q)+w_{s_0}(\alpha_q^{s_0} - \alpha^{s_0})^2+P_q(s_0)+o(\epsilon, n).
\end{align*}
Solve for $|\alpha_q^{s_0} - \alpha^{s_0}|$ and recall that
\[
g^{s_0}_q(\alpha^{s_0}_q)\le g^{s_0}(\alpha_q^{s_0}),
\]
(because the sum in definition of $g^s_q$ is over a smaller set $F_n^q$) to arrive at the conclusion.
\end{proof}

For the rest of the section, fix $s_0$ to satisfy $P(s_0)=0$ and define $\tilde t=\tilde t_q$ through
\[
P_q(s_0)+\inner{(a(\tilde t) - a(s_0))}{\alpha^{s_0}}=0.
\]
\begin{remark}\label{rem:astoP}
Notice that
\[
\inner{(a(\tilde t) - a(s_0))}{\alpha^{s_0}}=(\tilde t - s_0)\inner{(\log a^1_{e_{\ceil {s_0}}}, \dots, \log a^k_{e_{\ceil {s_0}}})}{\alpha^{s_0}}.
\]
Furthermore, from the definition of $\tilde t$,
\[
\tilde t - s_0=\frac{-P_q(s_0)}{\inner{(\log a^1_{e_{\ceil {s_0}}}, \dots, \log a^k_{e_{\ceil {s_0}}})}{\alpha^{s_0}}}.
\]

\end{remark}
In order to prove Theorem \ref{thm:main} we need to compare $s_0$ and $t_q$. By the above remark, in fact it suffices to compare $\tilde t$ and $t_q$. The next Lemma gives us a tool to do that.
\begin{lemma}\label{lem:MVT}
There are constants $0<c\le C<\infty$ such that
\[
c|P_q(\tilde t)|\le| t_q - \tilde t| \le C|P_q(\tilde t)|.
\]
\end{lemma}
\begin{proof}
It is standard to check that there are $0<b\le B<\infty$ such that for all $\epsilon>0$, $n$,
\[
b^{\epsilon n}\le \frac{\sum_{\bi\in \Sigma_{n,q}}\phi^{t+\epsilon}(A_\bi)}{\sum_{\bi\in \Sigma_{n,q}}\phi^t(A_\bi)}\le B^{\epsilon n}.
\]
It follows that there are $0<c\le C<\infty$ such that for all $t$ between $\tilde t$ and $t_q$, the absolute value of the left and right derivatives of $P_q$ at $t$ are all bounded from below by $c$ and above by $C$. The left and right derivatives exist at all points by convexity of $P_q$. Hence, recalling $P_q(t_q)=0$, the claim follows.
\end{proof}

In the remainder of the section, instead of writing down explicit constants, we will use the notation $O(-P_q(s_0))$ to mean a function of the form $C(-P_q(s_0))$ where the constant $C>0$ can be chosen to be independent of $q, n$ and $\epsilon$.
\begin{proposition}\label{prop:lower}
The quantity $t_q - \tilde t$ has a lower bound in terms  of $P_q(s_0)$, namely
\[
t_q - \tilde t\ge -O(-P_q(s_0)^{3/2}).
\]
\end{proposition}
\begin{proof}
Notice that by Lemma \ref{lem:aprox}, the definition of $\tilde t$, Remark \ref{rem:astoP} and Lemma \ref{lem:delta0}
\begin{align*}
P_q(\tilde t)&\ge g^{\tilde t}_q(\alpha^{s_0}_q)+o(\epsilon, n)=g^{s_0}_q(\alpha_q^{s_0})+\inner{(a(\tilde t) - a(s_0))}{\alpha^{s_0}_q}+o(\epsilon, n)\\
&=P_q(s_0)+\inner{(a(\tilde t) - a(s_0))}{\alpha^{s_0}} + \inner{(a(\tilde t) - a(s_0))}{(\alpha^{s_0}_q-\alpha^{s_0})}+o(\epsilon, n)\\
&\ge \inner{(a(\tilde t) - a(s_0))}{(\alpha_q^{s_0} - \alpha^{s_0})}+o(\epsilon, n).
\end{align*}
By Lemma \ref{lem:delta0} and Remark \ref{rem:astoP} this yields
\[
P_q(\tilde t)\ge -O(-P_q(s_0)^{3/2})+o(\epsilon, n).
\]
Finally, apply Lemma \ref{lem:MVT} and
 let $\epsilon\to 0$ and $n\to \infty$.
\end{proof}

\begin{lemma}\label{lem:delta3}
The distance between $\alpha_q^{s_0}$ and $\alpha^{\tilde t}_q$ is controlled by $P_q(s_0)$, namely
\[
|\alpha^{s_0}_q - \alpha^{\tilde t}_q|\le O((-P_q(s_0))^{1/2})+o(\epsilon, n).
\]
\end{lemma}
\begin{proof}
Notice first that by Lemma \ref{lem:concave},
\begin{equation}\label{eqn:ss1}
|\alpha^{\tilde t} - \alpha^{s_0}|\le \frac{|a(\tilde t) - a(s_0)|}{2w},
\end{equation}
where $w=w_{s_0}$. Using Lemma \ref{lem:com} and Remark \ref{rem:astoP}
\begin{align*}
g^{\tilde t}(\alpha^{\tilde t})&\le g^{s_0}(\alpha^{s_0})+L|a(\tilde t) - a(s_0)|\\
&=L|a(\tilde t) -a(s_0) |+ o(\epsilon, n)\\
&=O(-P_q(s_0)) + o(\epsilon, n).
\end{align*}
We now obtain from \eqref{eqn:concave} and the definition of $g^s_q$
\begin{align*}
w|\alpha^{\tilde t} - \alpha^{\tilde t}_q|^2&\le g^{\tilde t}(\alpha^{\tilde t}) - g^{\tilde t} (\alpha^{\tilde t}_q)\\
&\le g^{\tilde t}(\alpha^{\tilde t}) - g^{\tilde t}_q(\alpha^{s_0}_q)\\
&= g^{\tilde t}(\alpha^{\tilde t}) - g^{s_0}_q(\alpha^{s_0}_q) - \inner{(a(\tilde t) - a(s_0))}{\alpha^{s_0}_q}+o(\epsilon, n).
\end{align*}
These calculations combined amount to
\begin{equation}\label{eqn:delta2}
|\alpha^{\tilde t} - \alpha^{\tilde t}_q|\le O((-P_q(s_0))^{1/2})+o(\epsilon, n).
\end{equation}
Finally, through Lemma \ref{lem:delta0}, \eqref{eqn:ss1}and \eqref{eqn:delta2},
\[
|\alpha^{s_0}_q - \alpha^{\tilde t}_q|\le |\alpha^{s_0}_q - \alpha^{s_0}|+ |\alpha^{s_0} - \alpha^{\tilde t}|+|\alpha^{\tilde t} - \alpha^{\tilde t}_q|\le O((-P_q(s_0))^{1/2})+o(\epsilon, n).
\]

\end{proof}

\begin{proposition}\label{prop:upper}
The quantity $t_q - \tilde t$ has an upper bound in terms of $P_q(s_0)$, namely
\[
t_q - \tilde t\le O((-P_q(s_0))^{3/2}).
\]
\end{proposition}
\begin{proof}
Using Lemma \ref{lem:aprox} and the definition of $\tilde t$, Remark \ref{rem:astoP} and Lemmas \ref{lem:delta3} and \ref{lem:delta0}
\begin{align*}
P_q(\tilde t)&=g_q^{\tilde t}(\alpha^{\tilde t}_q)+o(\epsilon, n)=\inner{a(\tilde t) - a(s_0)}{\alpha^{\tilde t}_q} + g^{s_0}_q(\alpha^{\tilde t}_q)+o(\epsilon, n)\\
&\le \inner{a(\tilde t) - a(s_0)}{\alpha^{\tilde t}_q} + g^{s_0}_q(\alpha^{s_0}_q)+o(\epsilon, n)\\
&=\inner{a(\tilde t) - a(s_0)}{\alpha^{s_0}}+\inner{a(\tilde t) - a(s_0)}{(\alpha^{s_0}_q - \alpha^{s_0})} \\
&\quad  \quad+\inner{a(\tilde t) - a(s_0)}{(\alpha^{\tilde t}_q - \alpha^{s_0}_q)} +P_q(s_0)+o(\epsilon, n)\\
&\le O((-P_q(s_0))^{3/2})+o(\epsilon, n).
\end{align*}
Finally, apply Lemma \ref{lem:MVT} and let $\epsilon \to 0$ and $n\to \infty$.
\end{proof}

We are now ready to formulate the main theorem (in the diagonal case). Please recall the notation introduced in the beginning of the section. Denote
\begin{equation}\label{eq:z}
\begin{aligned}
Z(D) &= -\left(\lim_{h\searrow 0} \frac 1h (P_D(s_0) - P_D(s_0-h))\right)^{-1} \\
&= \frac{-1}{\inner{(\log a^1_{e_{\ceil {s_0}}}, \dots, \log a^k_{e_{\ceil {s_0}}})}{\alpha^{s_0}}}.
\end{aligned}
\end{equation}
\begin{theorem} \label{thm:diagonal}
Let $\Lambda$ be a self-affine set corresponding to an iterated function system $\{A_1+v_1, \dots, A_k+v_k\}$ with $\|A_i\|<\tfrac 12$ for all $i=1, \dots, k$, and let $\bq\in \Sigma$. Assume that all the matrices $A_i$ are diagonal. Then, for Lebesgue almost all $(v_1, \dots, v_k)\in \R^{dk}$,
\[
\dim_H \Lambda_q = \max_{D\in S_d} t_q^D.
\]
Moreover, if $\bq$ is not periodic then
\begin{equation} \label{eqn:res1}
\lim_{q\to \infty} \frac {\dim\Lambda - \dim \Lambda_q} {\min_{D\in S_d; P_D(s_0)=0} Z(D) \mu_D([\bq|_q])} =1,
\end{equation}
while if $\bq$ has period $\ell$ then
\begin{equation} \label{eqn:res2}
\lim_{q\to \infty} \frac {\dim\Lambda - \dim \Lambda_q} {\min_{D\in S_d; P_D(s_0)=0} Z(D) (1-\mu_D([\bq|_\ell])) \mu_D([\bq|_q])} =1.
\end{equation}

%\[
%\lim_{q\to \infty} \frac {s_0 - t_q^D} {\mu_D([\bq|_q])}=  \begin{cases}
%   \frac{1}{\inner{(\log a^1_{e_{\ceil {s_0}}}, \dots, \log a^k_{e_{\ceil {s_0}}})}{\alpha^{s_0}}}, & \bq\textrm{ is not periodic }\\
%    \frac{1 - p_{\bq|_\ell}}{\inner{(\log a^1_{e_{\ceil {s_0}}}, \dots, \log a^k_{e_{\ceil {s_0}}})}{\alpha^{s_0}}}, & \bq \textrm{ is periodic with period }\ell.
%\]
\end{theorem}
\begin{proof}
For a fixed $D$ the limit 
\[
\lim_{q\to \infty}\frac{s_0 - t_q^D}{Z(D)\mu_D([\bq|_q])}
\]
exists: the value comes from Lemma \ref{lem:onlypressure}, where the upper bound is from Proposition \ref{prop:lower} and Remark \ref{rem:astoP}, and the lower bound is obtained analogously, but using Proposition \ref{prop:upper}. To obtain the theorem we pass with $q$ to $\infty$ and for each $q$ use the $D$ for which $t_q^D$ is maximal.
\end{proof}

\begin{remark}\label{rem:exs}
Observe that in general situation we cannot write the usual formula `the dimension deficit divided by the measure of the hole converges to the left derivative of the pressure'. The reason: Consider an iterated function system as in \cite[Example 6.2]{KaenmakiVilppolainen10} with linear parts, say,
\[
\left ( \begin{matrix}
  4/9 & 0 \\
  0 & 1/9 \\
  \end{matrix}\right )\quad\textrm{and} \quad \left ( \begin{matrix}
  1/9 & 0 \\
  0 & 4/9 \\
  \end{matrix}\right ).
\]
Then $s_0=1/2$ and the collection of permutations which satisfies $P_D(s_0)=0$ consists of two elements, and the corresponding K\"aenm\"aki measures are the Bernoulli measures with weights $(1/3, 2/3)$ and $(2/3, 1/3)$, respectively. Now choose a very rapidly increasing sequence of natural numbers $(m_j)$ and set $\bq=(1^{m_1}2^{m_2}1^{m_3}...)$. Then the limits in \eqref{eqn:res1} and \eqref{eqn:res2} do not exist for either fixed $D$.
\end{remark}

\begin{remark}\label{rem:exs2}
However, the following shows that sometimes we can: Consider the case that $A_i$ are diagonal for all $i=1, \dots, k$, and, furthermore, the diagonal elements are in the same order in all of the matrices. Then the K\"aenm\"aki measure $\mu$ for the value of $s$ for which $P(s)=0$ is a Bernoulli measure with weights $(\phi^s(A_1), \dots, \phi^s(A_k))$ (by Lemma \ref{lem:measure}), and one can check that in Theorem \ref{thm:diagonal}, $\mu$ is the maximizing measure (or one of them, if there are many). Hence we obtain the statement of Theorem \ref{thm:main}
\[
\lim_{q\to \infty}\frac{\dim \Lambda - \dim \Lambda_q}{\mu([\bq|_q])}=  \begin{cases}
   \frac{-1}{\inner{(\log a^1_{\ceil {s_0}}, \dots, \log a^k_{\ceil {s_0}})}{\alpha^{s_0}}}, & \bq\textrm{ is not periodic }\\
    \frac{-1 + \mu([\bq|_\ell])}{\inner{(\log a^1_{\ceil {s_0}}, \dots, \log a^k_{\ceil {s_0}})}{\alpha^{s_0}}}, & \bq \textrm{ is periodic with period }\ell.
  \end{cases}
\]
\end{remark}

\begin{remark}\label{rem:exs3}
Fix some $\beta<\alpha<1/2$, and let $\gamma<\alpha, \beta$. Consider the iterated function system which has as the linear parts of the mappings
\[
A=\left (\begin{matrix}
  \alpha & \gamma \\
  0 & \beta \\
  \end{matrix}\right )\quad\textrm{and} \quad B= \left ( \begin{matrix}
  \beta & \gamma \\
  0 & \alpha \\
  \end{matrix}\right ).
\]
Then for $s<1$, $\phi^s(A^nB^n)$ grows like $\alpha^{2ns}$, whereas $\phi^s(B^nA^n)$ grows like $\alpha^{ns}\beta^{ns}$ so that there is an exponential gap between the values, due to the off-diagonal element. Our proof of Theorem \ref{thm:diagonal} depends on the exact connection between the singular value function and the Bernoulli measures given by the diagonal elements. Hence, despite the fact that according \cite[Theorem 2.6]{FalconerMiao07} the pressure only depends on the diagonal elements of $A$ and $B$, our proof does not easily extend to the upper triangular case.
\end{remark}

\section{The case of strong-Gibbs measures (Theorem \ref{thm:homogeneous})}\label{sec:abstract}

In this section, recall the assumptions that $s_0$ is such that $P(s_0)=0$, that the K\"aenm\"aki measure $\mu$ at $s_0$ is a strong-Gibbs measure, and given $q$, denote by $t_q$ the value where $P_q(t_q)=0$. Furthermore, we assume that the derivative $P'(s_0)$ exists. We do not assume that $P_q$ is differentiable, but since it is convex we know that left and right derivatives exist at all points. We know that $P_q$ is a convex function not
larger than $P$, which is also convex.

Let us begin with a simple lemma. Here by $f'(x-0)$ and $f'(x+0)$ we denote the left and right derivatives of $f$ at $x$.

\begin{lemma} \label{lem:convconv}
Let $P$ be a convex function. Let $P_q$ be a sequence of convex functions such that $P_q \leq P$ but $\lim_q P_q(s_0) = P(s_0)$. Then
\[
P'(s_0-0) \leq \lim_{q\to \infty} P_q'(s_0-0) \leq \lim_{q\to \infty} P_q'(s_0+0) \leq P'(s_0+0).
\]
\end{lemma}

\begin{proof}
It is enough to prove the first inequality: the second is immediate from convexity, and the third can be proved analogously to the first. Assume to the contrary, that there exists $\varepsilon>0$, and we can choose a subsequence of convex functions $P_q\leq P$ with $P_q(s_0)\to P(s_0)$, such that
\[
P_q'(s_0-0) < P'(s_0-0) - \varepsilon.
\]
As $P_q$ is convex, $P_q'(s-0) \leq P_q'(s_0-0)$ for all $s<s_0$. On the other hand,
\[
P'(s_0-0) = \lim_{s\nearrow s_0} P'(s-0),
\]
hence there exists $\delta>0$ depending only on $P$ such that
\[
P'(s-0) > P'(s_0-0) - \varepsilon/2
\]
for all $s> s_0-\delta$. Hence, decreasing $\delta>0$ further if necessary
\begin{align*}
P_q(s_0-\delta) &\leq P_q(s_0) - \delta P_q'(s_0-0) \\
&\leq \left[P(s_0) - \delta P'(s_0-\delta-0)\right] + P_q(s_0) - P(s_0) + \delta\varepsilon/2 \\
&\leq P(s_0-\delta) + P_q(s_0) - P(s_0) + \delta\varepsilon/4
\end{align*}
Therefore, choosing $q$ so large that $P_q(s_0) > P(s_0) - \delta\varepsilon/4$, we obtain $P_q(s_0-\delta) > P(s_0-\delta)$, which is a contradiction.
\end{proof}

Relying on Lemma \ref{lem:onlypressure}, we wish to understand $s_0 - t_q$ in terms of $P_q(s_0)$. That is the content of the following lemma.

\begin{lemma} \label{lem:convconv2}
Let $P$ be a convex function. Let $P_q$ be a sequence of convex functions such that $P_q \leq P$, $\lim_q P_q(s_0) = P(s_0)=0$, and $\lim_q P_q'(s_0-0) = P'(s_0-0)$. Then

\[
\lim_{q\to \infty} \frac  {-P_q(s_0)} {s_0-t_q} = -P'(s_0-0).
\]
\end{lemma}
\begin{proof}
We have

\[
P_q(s_0) = \int_{t_q}^{s_0} P_q'(s-0) ds.
\]

As $P_q'(s-0) \leq P_q'(s_0-0)$ for all $s<s_0$, the upper bound follows immediately.

For the lower bound, assume that it fails: for a subsequence of $P_q$ we have

\[
\frac  {-P_q(s_0)} {s_0-t_q} < -P'(s_0-0) - \varepsilon.
\]

Then, necessarily,
\[
P_q'(t_q-0) < P'(s_0-0) - \varepsilon.
\]
Hence, for all $s<t_q$ we have

\begin{equation} \label{eqn:convconv1}
P_q(s) \geq -(t_q-s) (P'(s_0-0)-\varepsilon).
\end{equation}

On the other hand, like in the previous lemma, we can find $\delta>0$ not depending on $q$ such that

\[
P'(s-0) > P'(s_0-0) - \varepsilon/2
\]
for all $s> s_0-\delta$. Thus,

\begin{equation} \label{eqn:convconv2}
P(s_0-\delta) \leq - \delta (P'(s_0-0)-\varepsilon/2).
\end{equation}

Comparing \eqref{eqn:convconv1} with \eqref{eqn:convconv2} we see that choosing $q$ such that $t_q$ is so close to $s_0$ that
\[
(s_0-t_q - \delta) (P'(s_0-0)-\varepsilon) > - \delta (P'(s_0-0) -\varepsilon/2),
\]
that is
\[
s_0-t_q < - \frac {\delta \varepsilon} {2(P'(s_0-0)-\varepsilon)},
\]
then we get $P_q(s_0-\delta) > P(s_0-\delta)$, a contradiction.
\end{proof}

Under the assumption that $P'(s_0)$ exists, by Lemma \ref{lem:convconv} we can apply Lemma \ref{lem:convconv2}. The statement of Theorem \ref{thm:homogeneous} is now an immediate corollary of Lemmas \ref{lem:convconv2} and \ref{lem:onlypressure}, and Theorem \ref{thm:ferguson pollicott}.

\begin{remark}
The assumptions of Theorem \ref{thm:homogeneous} may look difficult to satisfy, but there are at least two classes of systems for which the K\"aenm\"aki measure is strong-Gibbs.
\begin{itemize}
\item[(1)]Homogeneous case: Assume that all the matrices $A_i$ are powers of one matrix $A$. To demonstrate our result, consider the simplest case where $A_i=A$ for all $i$. Then the K\"aenm\"aki measure is a Bernoulli measure with equal weights by Lemma \ref{lem:measure} so that, in particular, it is strong-Gibbs. Writing $\sigma_1, \dots, \sigma_d$ for the singular values of $A$ and assuming that the dimension $s_0$ of $\Lambda$ is not an integer, one can obtain
\[
P'(s_0)=\log \sigma_{\ceil{s_0}}.
\]
\item[(2)]Dominated case: Assume that $d=2$ and the cocycle generated by matrices $A_i$ is dominated, that is, there exist $C>0$,  $0<\tau<1$ such that for all $n$ and $\bi\in \Sigma_n$, 
\[
\frac{\det(A_\bi)}{|A_\bi|^2}\le C\tau^{n}.
\]
It is proved in \cite{BaranyKaenmakiMorris} that also in this case the K\"aenm\"aki measure satisfies the strong-Gibbs assumption, and if $s_0$ is not an integer then $P'(s_0)$ is well defined. The dominated cocycles are an open subset of $GL(2,\R)$-cocycles, we refer the reader to \cite{BaranyKaenmakiMorris} for the discussion.
\end{itemize}
For more on the $s$-semiconformality of K\"aenm\"aki measures, see \cite{KaenmakiVilppolainen10}.
\end{remark}

\begin{remark}
As one can see in Lemma \ref{lem:convconv2}, in both examples presented above the assertion of our theorem stays true for integer $s_0$ (with $P'(s_0)$ replaced by $P'(s_0-0)$). Indeed, while the singular value pressure is nondifferentiable at integer points because of nondifferentiability at those points of the definition of singular value function, the assumptions of Lemma \ref{lem:convconv2} are satisfied (for those examples) at integer points as well.
\end{remark}

\bibliographystyle{plain}
\bibliography{vaitbib}

\end{document}